\documentclass[11pt,reqno]{amsart}
\usepackage{fullpage}
\usepackage{mathrsfs,amssymb,graphicx,verbatim,amsmath,amsfonts}
\usepackage{paralist}
\usepackage[breaklinks,pdfstartview=FitH]{hyperref}
\usepackage{upgreek}
\usepackage{tikz}
\usepackage[mathscr]{euscript}

\hypersetup{ocgcolorlinks=true,allcolors=testc}
\hypersetup{
     colorlinks   = true,
     citecolor    = black
}
\hypersetup{linkcolor=black}
\hypersetup{urlcolor=black}

\renewcommand{\leq}{\leqslant}
\renewcommand{\geq}{\geqslant}
\renewcommand{\setminus}{\smallsetminus}
\renewcommand{\gamma}{\upgamma}
\allowdisplaybreaks

\usepackage{esint}
\usepackage[autostyle]{csquotes}

\renewcommand{\pi}{\uppi}

\newcommand{\e}{\varepsilon}
\newcommand{\R}{\mathbb R}

\newtheorem{theorem}{Theorem}
\newtheorem{lemma}[theorem]{Lemma}
\newtheorem{proposition}[theorem]{Proposition}

\newtheorem{corollary}[theorem]{Corollary}

\theoremstyle{remark}
\newtheorem{remark}[theorem]{Remark}

\newtheorem{question}[theorem]{Question}
\renewcommand{\tau}{\uptau}

\renewcommand{\xi}{\upxi}
\renewcommand{\rho}{\uprho}

\newcommand{\C}{\mathbb C}

\newcommand{\N}{\mathbb N}

\newcommand{\eqdef}{\stackrel{\mathrm{def}}{=}}

\renewcommand{\theta}{\uptheta}
\renewcommand{\lambda}{\uplambda}

\renewcommand{\gamma}{\upgamma}
\renewcommand{\beta}{\upbeta}
\renewcommand{\alpha}{\upalpha}
\renewcommand{\kappa}{\upkappa}
\renewcommand{\psi}{\uppsi}
\renewcommand{\rho}{\uprho}
\renewcommand{\delta}{\updelta}
\renewcommand{\pi}{\uppi}
\renewcommand{\omega}{\upomega}
\renewcommand{\sigma}{\upsigma}
\renewcommand{\phi}{\upphi}

\renewcommand{\eta}{\upeta}

\renewcommand{\kappa}{\upkappa}
\renewcommand{\mu}{\upmu}
\renewcommand{\nu}{\upnu}
\renewcommand{\pi}{\uppi}
\renewcommand{\zeta}{\upzeta}

\newcommand{\mb}{\mathbb}
\newcommand{\mr}{\mathrm}
\newcommand*\diff{\mathop{}\!\mathrm{d}}

\begin{document}

\title{On extremal sections of subspaces of $L_p$}

\author{Alexandros Eskenazis}
\address{Mathematics Department\\ Princeton University\\ Fine Hall, Washington Road, Princeton, NJ 08544-1000, USA}
\email{ae3@math.princeton.edu}
\thanks{This work was completed while the author was in residence at the Institute for Pure \& Applied Mathematics at UCLA for the long program on Quantitative Linear Algebra. He would like to thank the organizers of the program for the excellent working conditions. He was also supported in part by the Simons Foundation.}

\begin{abstract}
Let $m,n\in\N$ and $p\in(0,\infty)$. For a finite dimensional quasi-normed space $X=(\R^m, \|\cdot\|_X)$, let
$$B_p^n(X) = \Big\{ (x_1,\ldots,x_n)\in\big(\R^{m}\big)^n: \ \sum_{i=1}^n \|x_i\|_X^p \leq 1\Big\}.$$
We show that for every $p\in(0,2)$ and $X$ which admits an isometric embedding into $L_p$, the function
$$S^{n-1} \ni \theta = (\theta_1,\ldots,\theta_n) \longmapsto \Big| B_p^n(X) \cap\Big\{(x_1,\ldots,x_n)\in \big(\R^{m}\big)^n: \ \sum_{i=1}^n \theta_i x_i=0 \Big\} \Big|$$
is a Schur convex function of $(\theta_1^2,\ldots,\theta_n^2)$, where $|\cdot|$ denotes Lebesgue measure. In particular, it is minimized when $\theta=\big(\frac{1}{\sqrt{n}},\ldots,\frac{1}{\sqrt{n}}\big)$ and maximized when $\theta=(1,0,\ldots,0)$. This is a consequence of a more general statement about Laplace transforms of norms of suitable Gaussian random vectors which also implies dual estimates for the mean width of projections of the polar body $\big(B_p^n(X)\big)^\circ$ if the unit ball $B_X$ of $X$ is in Lewis' position. Finally, we prove a lower bound for the volume of projections of $B_\infty^n(X)$, where $X=(\R^m,\|\cdot\|_X)$ is an arbitrary quasi-normed space.
\end{abstract}

\maketitle

{\footnotesize

}


\section{Introduction}
For a quasi-normed space $(X,\|\cdot\|_X)$, $n\in\N$ and $p\in(0,\infty)$, the $\ell_p^n$ power of $X$, denoted $\ell_p^n(X)$, is the space $X^n$ equipped with the quasi-norm
\begin{equation} \label{eq:ellpn(x)}
\|(x_1,\ldots,x_n)\|_{\ell_p^n(X)} \eqdef \Big( \sum_{i=1}^n \|x_i\|_X^p\Big)^{\frac{1}{p}},
\end{equation}
where $(x_1,\ldots,x_n)\in X^n$. In particular, if $m\in\N$ and $X=(\R^m,\|\cdot\|_X)$ is an $m$-dimensional space, then $\ell_p^n(X)$ is an $mn$-dimensional quasi-normed space, whose unit ball we denote by 
\begin{equation} \label{eq:bpn(x)}
B_p^n(X) \eqdef \big\{x\in\R^{mn}: \ \|x\|_{\ell_p^n(X)}\leq1\big\}.
\end{equation}
Similarly, we define $\ell_\infty^n(X)$ to be the quasi-normed space with unit ball $B_\infty^n(X) \eqdef B_X^n$, where $B_X\subseteq\R^m$ is the unit ball of $X$. If $X$ is a normed space and $p\geq1$, then $B_p^n(X)$ is a convex body, though in general it is always a star body. Extremal sections of such bodies have been thoroughly studied in the literature via a variety of analytic, geometric and probabilistic techniques (see the monograph \cite{Kol05} of Koldobsky for an exposition of some of these). In this paper, we will be interested in sections of $B_p^n(X)$, $p\in(0,\infty]$, with {\it block} hyperplanes, i.e. subspaces of the form
\begin{equation} \label{eq:Htheta}
H_\theta \eqdef \theta^\perp \otimes \R^m = \Big\{ x=(x_1,\ldots,x_n) \in \big(\R^{m}\big)^n: \ \sum_{i=1}^n \theta_i x_i = 0\Big\},
\end{equation}
where $\theta$ is a unit vector in $\R^n$. We refer the reader to Remark \ref{rem:block} for a discussion explaining the necessity of this particular choice of subspaces (see also Lemma \ref{lem:invariance}). The most well-understood case is, naturally, the Hilbert space case $X=\ell_2^m$, where it is known that for every unit vector $\theta$ in $\R^n$,
\begin{equation} \label{eq:p<2}
p\in(0,2] \ \ \Longrightarrow \ \ \big|B_p^n(\ell_2^m)\cap H_\theta\big| \leq \big|B_p^{n-1}(\ell_2^m)\big|
\end{equation}
and
\begin{equation} \label{eq:p>2}
p\in[2,\infty] \ \ \Longrightarrow \ \ \big|B_p^{n-1}(\ell_2^m)\big| \leq \big|B_p^n(\ell_2^m) \cap H_\theta\big|,
\end{equation}
where $|\cdot|$ denotes the Lebesgue measure. The case $p=\infty$ of \eqref{eq:p>2} was proven by Vaaler in \cite{Vaa79}, though the case $m=1$ of hyperplane sections of the unit cube had first been shown by Hadwiger \cite{Had72} (see also \cite{Hen79}). Afterwards, Meyer and Pajor \cite[p.~116]{MP88} proved \eqref{eq:p>2} for $p\in(2,\infty)$ and \eqref{eq:p<2} for $p\in[1,2)$. Finally, the $p\in(0,1)$ case of \eqref{eq:p<2} was settled by Caetano in \cite{Cae92} for $m=1$ (see also \cite{Bar95}) and by Barthe \cite{Bar01} for general $m\in\N$. We note in passing that \cite{Vaa79}, \cite{MP88} and \cite{Bar01} contain similar results for sections of $B_p^n(\ell_2^m)$ with \mbox{block subspaces of arbitrary codimension.}

The study of reverse inequalities to \eqref{eq:p<2} and \eqref{eq:p>2} is notoriously more involved, even for $X=\R$. The maximal hyperplane section of the unit cube $B_\infty^n(\R)$ was shown to be $B_\infty^n(\R)\cap\big(\frac{1}{\sqrt{2}},\frac{1}{\sqrt{2}},0,\ldots,0\big)$ in Ball's work \cite{Bal86}. Moreover, in \cite{MP88}, Meyer and Pajor showed that the minimal hyperplane section of the crosspolytope $B_1^n(\R)$ is $B_1^n(\R)\cap\big(\frac{1}{\sqrt{n}},\ldots,\frac{1}{\sqrt{n}}\big)^\perp$, a result which was later extended by Koldobsky \cite{Kol98} who proved that
\begin{equation} \label{eq:koldobsky}
p\in(0,2] \ \ \Longrightarrow \ \ \Big| B_p^n(\R)\cap\Big(\frac{1}{\sqrt{n}},\ldots,\frac{1}{\sqrt{n}}\Big)^\perp\Big| \leq \big|B_p^n(\R)\cap \theta^\perp\big|,
\end{equation}
for every unit vector $\theta$ in $\R^n$. The question of determining the maximal hyperplane sections of $B_p^n(\R)$ for $p\in(2,\infty)$ remains open (see \cite[Proposition~5]{Ole03} for a related observation).

The aforementioned reverse inequalities of Ball, Meyer--Pajor and Koldobsky for $B_p^n(\R)$ have well-studied {\it complex} counterparts for the unit balls of the complex $\ell_p^n(\C)$ spaces, which can be isometrically identified with $\ell_p^n(\ell_2^2)$ following the notation \eqref{eq:ellpn(x)}. Recall that a complex hyperplane of $\C^n$ is a subspace of the form $w^\perp = \{z\in \C^n: \ \langle z,w\rangle = 0\}$, where $\langle\cdot,\cdot\rangle$ denotes the Hermitian inner product on $\C^n$ and $w\in\C^n$. Oleszkiewicz and Pe\l czy\'nski proved in \cite{OP00} that the maximal complex hyperplane section of $B_\infty^n(\C)$ is $B_\infty^n(\C)\cap \big(\frac{1}{\sqrt{2}},\frac{1}{\sqrt{2}},0,\ldots,0\big)^\perp$ and Koldobsky and Zymonopoulou \cite{KZ03} showed that the minimal complex hyperplane section of $B_p^n(\C)$ is $B_p^n(\C) \cap \big(\frac{1}{\sqrt{n}},\ldots,\frac{1}{\sqrt{n}}\big)^\perp$ for every $p\in(0,2]$, in perfect analogy to the real case.

The main purpose of the present article is to unify and extend these results of \cite{MP88}, \cite{Kol98} and \cite{KZ03}, by exhibiting a wide class of (non-Hilbertian) quasi-normed spaces $X=(\R^m,\|\cdot\|_X)$, for which the extremal sections of $B_p^n(X)$, $p\in(0,2)$, with subspaces of the form \eqref{eq:Htheta} can be determined. To better state our main result, we will make use of the {\it Schur majorization} ordering. A vector $\alpha=(\alpha_1,\ldots,\alpha_n)\in\R^n$ is said to be majorized by a vector $\beta=(\beta_1,\ldots,\beta_n)\in\R^n$, denoted $\alpha\preceq \beta$, if their nonincreasing rearrangements $\alpha_1^\ast\geq\ldots\geq \alpha_n^\ast$ and $\beta_1^\ast\geq\ldots\geq \beta_n^\ast$ satisfy the inequalities
\begin{equation}
\sum_{j=1}^k \alpha_j^\ast \leq \sum_{j=1}^k \beta_j^\ast \ \ \mbox{for every } k\in\{1,\ldots,n-1\} \ \ \mbox{and} \ \ \sum_{j=1}^n \alpha_j = \sum_{j=1}^n \beta_j.
\end{equation}
For instance, if a vector $(\alpha_1,\ldots,\alpha_n)\in[0,\infty)^n$ satisfies $\sum_{i=1}^n \alpha_i=1$, then
\begin{equation}
\Big(\frac{1}{n},\ldots,\frac{1}{n}\Big) \preceq (\alpha_1,\ldots,\alpha_n) \preceq (1,0,\ldots,0).
\end{equation}
Our main result concerning the volume of sections of $B_p^n(X)$ is the following.

\begin{theorem} \label{thm:Lpsubspaces}
Fix $m,n\in\N$, $p\in(0,2]$ and let $X=(\R^m,\|\cdot\|_X)$ be an $m$-dimensional quasi-normed space which admits an isometric embedding into $L_p$. Then, for every unit vectors $\theta=(\theta_1,\ldots,\theta_n)$ and $\phi=(\phi_1,\ldots,\phi_n)$ in $\R^n$, we have
\begin{equation} \label{eq:Lpsubspaces}
(\theta_1^2,\ldots,\theta_n^2) \preceq (\phi_1^2,\ldots,\phi_n^2) \ \ \ \Longrightarrow \ \ \ \big| B_{p}^n(X) \cap H_\theta\big| \leq \big|B_p^n(X) \cap H_\phi\big|.
\end{equation}
In particular, for every unit vector $\theta$ in $\R^n$, we have
\begin{equation} \label{eq:Lpsubspacesconsequence}
\big|B_{p}^{n}(X)\cap H_\mathrm{diag} \big| \leq \big|B_{p}^{n}(X)\cap H_\theta \big| \leq \big|B_{p}^{n-1}(X)\big|,
\end{equation}
where $H_\mathrm{diag} = \big\{ (x_1,\ldots,x_n) \in\big(\R^m\big)^n: \ \sum_{i=1}^n x_i=0\big\}$.
\end{theorem}

Theorem \ref{thm:Lpsubspaces} implies the results \eqref{eq:p<2} of Meyer and Pajor, \eqref{eq:koldobsky} of Koldobsky and its complex counterpart due to Koldobsky and Zymonopoulou since $L_2$ admits an isometric embedding into $L_p$ for every $p\in(0,\infty)$ (see \cite{Kad58} and \cite{BDCK66}). However, it also significantly extends those results. For instance, by \cite{Kad58} and \cite{BDCK66}, $L_q$ admits an isometric embedding into $L_p$ for every $0<p\leq q\leq 2$, therefore for every $m,n\in\N$ and unit vector $\theta$ in $\R^n$, we have
\begin{equation} \label{eq:lplq}
\big|B_{p}^{n}(\ell_q^m)\cap H_\mr{diag} \big| \leq \big|B_{p}^{n}(\ell_q^m)\cap H_\theta \big| \leq \big|B_{p}^{n-1}(\ell_q^m)\big|
\end{equation}
(see also the discussion preceeding and following Question \ref{q:lowdim} below). Since finite dimensional normed subspaces of $L_1$ are exactly those spaces $X$ such that the polar $B_X^\circ$ of their unit ball $B_X$ is a zonoid (see \cite{Bol69}), it follows that
\begin{equation} \label{eq:zonoid}
\big|B_{p}^{n}(X)\cap H_\mathrm{diag} \big| \leq \big|B_{p}^{n}(X)\cap H_\theta \big| \leq \big|B_{p}^{n-1}(X)\big|,
\end{equation}
whenever $B_X^\circ$ is a zonoid and $p\in(0,1]$. To illustrate the rich repertoire of examples contained in this class of spaces, we mention that by the work \cite{Sch75} of Schneider, for every $m\in\N$ there exists a zonoid $A\subseteq\R^m$ whose polar body $A^\circ$ is also a zonoid, yet $A$ is not an ellipsoid. Therefore, for every $m\in\N$, there exists an $m$-dimensional Banach space $X=(\R^m,\|\cdot\|_X)$ so that \eqref{eq:zonoid} is valid both for $X$ and its dual $X^\ast$ when $p\in(0,1]$, yet $X$ is not isometric to Hilbert space. Finally, since every two-dimensional symmetric convex body is a zonoid (see, e.g., \cite[Corollary~6.8]{Kol05}), we deduce that \eqref{eq:zonoid} also holds true for a general normed space of the form $X=(\R^2,\|\cdot\|_X)$ and all $p\in(0,1]$. For further extremal properties of sections and projections of unit balls of subspaces of $L_p$ spaces, we refer the reader to \cite{Bal91b}, \cite{Bar98}, \cite{LYZ04}, \cite{LHX18} and the references therein.

The proof of Theorem \ref{thm:Lpsubspaces} is probabilistic and builds upon \cite[Corollary~19]{ENT16}, where the Schur monotonicity statement of Theorem \ref{thm:Lpsubspaces} was verified for $X=\R$. The new ingredient needed to obtain Theorem \ref{thm:Lpsubspaces} is a classical Banach space theoretic result of Lewis (see \cite{Lew78} and Theorem \ref{thm:lewis} below) on the isometric structure of subspaces of $L_p$. As in \cite{ENT16} (see also \cite{BGMN05}), Theorem \ref{thm:Lpsubspaces} will be a consequence of a more general comparison for Laplace transforms of norms of suitable Gaussian random vectors which also implies dual estimates for the mean width of projections of the polar body $\big(B_p^n(X)\big)^\circ$ (see Theorem \ref{thm:Laplace} and Corollary \ref{cor:meanwidth} below).

In view of Theorem \ref{thm:Lpsubspaces}, a number of natural problems pose themselves. Perhaps most naturally, one would ask whether an analogue of Theorem \ref{thm:Lpsubspaces} holds true when $p\in(2,\infty]$ and $X$ is isometric to a subspace of $L_p$. We explain in Section \ref{sec:4} below that the left hand side of inequality \eqref{eq:Lpsubspacesconsequence} cannot have such an analogue even for $p=\infty$. However, we obtain a weak reverse inequality to the right hand side of \eqref{eq:Lpsubspacesconsequence} for $p=\infty$, where sections are replaced by projections but is valid for arbitrary symmetric compact sets (see also Question \ref{q:p>2} below). Recall that any separable Banach space embeds isometrically into $L_\infty$. Therefore, any symmetric convex body of the form $K^n$, where $K\subseteq\R^m$, can be identified with $B_\infty^n(X)$ for the normed space $X=(\R^m,\|\cdot\|_X)$ with unit ball $K$.

\begin{proposition} \label{prop:liako}
Fix $m,n\in\N$. For every symmetric compact set $K\subseteq\R^m$ and unit vector $\theta$ in $\R^n$, we have
\begin{equation}
\big| \mathrm{Proj}_{H_\theta}(K^n)\big| \geq |K|^{n-1}.
\end{equation}
\end{proposition}

The proof of Theorem \ref{thm:Lpsubspaces} in presented in Section \ref{sec:2} and the proof of Proposition \ref{prop:liako} is in Section \ref{sec:3}. Finally, we conclude with some additional remarks and open problems in Section \ref{sec:4}.

\subsection*{Acknowledgements} I would like to thank Franck Barthe, Apostolos Giannopoulos, Olivier Gu\'edon, Assaf Naor and the anonymous referees for constructive feedback on this work. I am also very grateful to Tomasz Tkocz for many helpful discussions.


\section{Proof of Theorem \ref{thm:Lpsubspaces}} \label{sec:2}

Before presenting the proof of Theorem \ref{thm:Lpsubspaces}, we make a simple observation which implies that the conclusion of the theorem is invariant under isometries of the space $X=(\R^m,\|\cdot\|_X)$.

\begin{lemma} \label{lem:invariance}
Fix $m,n\in\N$, $p\in(0,\infty]$ and let $X=(\R^m,\|\cdot\|_X)$ be a quasi-normed space. Then, for every invertible linear operator $T:\R^m\to\R^m$ and every unit vector $\theta$ in $\R^n$, we have
\begin{equation} \label{eq:invariance0}
|B_p^n(TX)\cap H_\theta| = \mr{det}(T)^{n-1} |B_p^n(X)\cap H_\theta|,
\end{equation}
where $TX=(\R^m,\|\cdot\|_{TX})$ and $\|y\|_{TX} = \|T^{-1}y\|_X$ for $y\in\R^m$. In particular, for every unit vectors $\theta,\phi$ in $\R^m$, we have
\begin{equation} \label{eq:invariance}
\frac{|B_p^n(TX)\cap H_\theta|}{|B_p^n(TX)\cap H_\phi|} = \frac{|B_p^n(X)\cap H_\theta|}{|B_p^n(X)\cap H_\phi|}.
\end{equation}
\end{lemma}

\begin{proof}
The proof is a simple linear algebra exercise. Consider the natural mapping $\overline{T} = I_n\otimes T : \R^{mn}\to \R^{mn}$ given by $\overline{T}(x_1,\ldots,x_n) = (Tx_1,\ldots,Tx_n)$. Then, one can easily check that $\overline{T}H_\theta = H_\theta$ and furthermore, since $T(B_X)=B_{TX}$, that
\begin{equation}
B_p^n(TX) = \overline{T}B_p^n(X).
\end{equation}
Therefore,
\begin{equation} \label{eq:invar1}
|B_p^n(TX)\cap H_\theta|= |\overline{T}\big(B_p^n(X)\cap H_\theta\big)| =\mr{det}(\overline{T}|_{H_\theta}) |B_p^n(X)\cap H_\theta|.
\end{equation}
Since $\overline{T}|_{H_\theta} = I_{\theta^\perp} \otimes T$, we deduce that
\begin{equation} \label{eq:invar2}
\mr{det}(\overline{T}|_{H_\theta}) = \mr{det}(I_{\theta^\perp}\otimes T) = \mr{det}(T)^{n-1},
\end{equation}
which along with \eqref{eq:invar1} implies \eqref{eq:invariance0}.
\end{proof}

In convex geometric terminology, Lemma \ref{lem:invariance} implies that the conclusion of Theorem \ref{thm:Lpsubspaces} is independent of the {\it position} of the unit ball $B_X$ of $X=(\R^m,\|\cdot\|_X)$. The starting point of the proof of Theorem \ref{thm:Lpsubspaces} is a classical result due to Lewis \cite{Lew78}, according to which the unit ball of every finite dimensional subspace of $L_p$ can be put in a special position, which is called {\it Lewis' position}. Recall that a Borel measure $\mu$ on the unit sphere $S^{m-1}$ on $\R^m$ is called isotropic if
\begin{equation} \label{eq:isotro}
\|x\|_2 = \Big(\int_{S^{m-1}} \langle x,\theta\rangle^2 \diff\mu(\theta)\Big)^{1/2},
\end{equation}
for every $x\in\R^m$. The following theorem was proven by Lewis for $p\in[1,\infty)$ in \cite{Lew78} and extended to the whole range $p\in(0,\infty)$ by Schechtman and Zvavitch \cite{SZ01}.

\begin{theorem}[Lewis, Schechtman--Zvavitch] \label{thm:lewis}
Fix $m\in\N$ and $p\in(0,\infty)$. An $m$-dimensional quasi-normed space $X=(\R^m,\|\cdot\|_X)$ embeds isometrically into $L_p$ if and only if there exists an isotropic measure $\mu$ on the unit sphere $S^{m-1}$ and an invertible linear transformation $U\colon\R^m\to\R^m$ such that
\begin{equation} \label{eq:lewis}
\|Ux\|_X= \|x\|_{p,\mu} \eqdef \Big(\int_{S^{m-1}} |\langle x,\theta\rangle|^p \diff\mu(\theta)\Big)^{1/p},
\end{equation}
for every $x\in\R^m$.
\end{theorem}

This formulation of Lewis' theorem is taken from \cite{LYZ05}. According to Lemma \ref{lem:invariance}, for the rest of the proof of Theorem \ref{thm:Lpsubspaces}, we can freely assume that the subspace of $L_p$ at hand is $X = (\R^m,\|\cdot\|_{p,\mu})$, in which case we will write $X= X_p(\mu)$ and say that $B_X$ is in Lewis' position. We refer to the survey \cite{JS01} for an in-depth account of finite dimensional subspaces of $L_p$ including Lewis' theorem.

As first observed by Barthe, Gu\'edon, Mendelson and Naor in \cite{BGMN05}, inequalities about volumes of sections such as \eqref{eq:Lpsubspacesconsequence} can be formally obtained by the comparison of Laplace transforms of suitable Gaussian random vectors. In our case, we will prove the following stronger theorem.

\begin{theorem} \label{thm:Laplace}
Fix $m,n\in\N$, $p\in(0,2]$ and let $X=X_p(\mu)$ be an $m$-dimensional subspace of $L_p$ such that $B_X$ is in Lewis' position. For a unit vector $a$ in $\R^n$, let $G_a$ be a standard Gaussian random vector on the subspace $H_a$ of $\R^{mn}$. Then for every $\lambda\in(0,\infty)$ and unit vectors $\theta=(\theta_1,\ldots,\theta_n)$ and $\phi=(\phi_1,\ldots,\phi_n)$ in $\R^n$, we have
\begin{equation} \label{eq:Laplace}
(\theta_1^2,\ldots,\theta_n^2) \preceq (\phi_1^2,\ldots,\phi_n^2) \ \ \ \Longrightarrow \ \ \ \mb{E} e^{-\lambda \|G_\theta\|^p_{\ell_p^n(X)\cap H_\theta}} \leq \mb{E} e^{-\lambda \|G_\phi\|^p_{\ell_p^n(X)\cap H_\phi}}.
\end{equation}
\end{theorem}

Using Theorem \ref{thm:Laplace}, one can additionally derive the following inequality.

\begin{corollary} \label{cor:meanwidth}
Fix $m,n\in\N$, $q\in[2,\infty]$ and let $X=X_p(\mu)$ be an $m$-dimensional subspace of $L_p$, where $\frac{1}{p}+\frac{1}{q}=1$, whose unit ball $B_X$ is in Lewis' position. Then for every unit vectors $\theta=(\theta_1,\ldots,\theta_n)$ and $\phi=(\phi_1,\ldots,\phi_n)$ in $\R^n$, we have
\begin{equation}
(\theta_1^2,\ldots,\theta_n^2) \preceq (\phi_1^2,\ldots,\phi_n^2) \ \ \ \Longrightarrow \ \ \ w\big(\mr{Proj}_{H_\theta}(B_q^n(X^\ast))\big) \geq w\big(\mr{Proj}_{H_\phi}(B_q^n(X^\ast))\big),
\end{equation}
where $w(A)$ is the mean width of the set $A$.
\end{corollary}

The derivation of volume and mean width inequalities from the comparison of Laplace transforms was explained in full detail in \cite[Section~6]{ENT16} in the scalar case $X=\R$. The necessary modifications for the vector-valued case treated here are only symbolic, yet we include them for the readers' convenience.

\medskip

\noindent {\it Theorem \ref{thm:Laplace} implies Theorem \ref{thm:Lpsubspaces}.} Let $\theta, \phi\in S^{n-1}$ such that $(\theta_1^2,\ldots,\theta_n^2) \preceq (\phi_1^2,\ldots,\phi_n^2)$ and $X=(\R^m,\|\cdot\|_X)$ an $m$-dimensional subspace of $L_p$. By Lemma \ref{lem:invariance}, the conclusion of Theorem \ref{thm:Lpsubspaces} is invariant under isometries and thus, by Theorem \ref{thm:lewis}, we can assume that $B_X$ is in Lewis' position. Then, by Theorem \ref{thm:Laplace}, for every $\lambda\in(0,\infty)$ we have
\begin{equation} \label{eq:needtoint}
\mb{E} e^{-\lambda \|G_\theta\|^p_{\ell_p^n(X)\cap H_\theta}} \leq \mb{E} e^{-\lambda \|G_\phi\|^p_{\ell_p^n(X)\cap H_\phi}}.
\end{equation}
Recall that an infinitely differentiable function $h:(0,\infty)\to\R_+$ is called completely monotonic if $(-1)^nh^{(n)}(t)\geq0$ for every $t\in(0,\infty)$ and every $n\in\N$. A classical theorem of Bernstein (see, e.g., \cite{Fel71}) characterizes completely monotonic functions as those who can be expressed as the Laplace transform of a Borel measure on $(0,\infty)$. Integrating \eqref{eq:needtoint} with respect to an arbitrary Borel measure $\mu$ on $(0,\infty)$, we deduce that
\begin{equation}
\mb{E} \int_0^\infty e^{-\lambda \|G_\theta\|^p_{\ell_p^n(X)\cap H_\theta}}\diff\mu(\lambda) \leq \mb{E} \int_0^\infty e^{-\lambda \|G_\phi\|^p_{\ell_p^n(X)\cap H_\phi}} \diff\mu(\lambda),
\end{equation}
which, by Bernstein's theorem, is equivalent to the validity of the inequality
\begin{equation} \label{eq:completelymonotonic}
\mb{E}h\big(\|G_\theta\|^p_{\ell_p^n(X)\cap H_\theta}\big) \leq \mb{E}h\big(\|G_\phi\|^p_{\ell_p^n(X)\cap H_\phi}\big)
\end{equation}
for every completely monotonic function $h:(0,\infty)\to\R$. A straightforward computation shows that the function $h(t)=t^{-\alpha/p}$ is completely monotonic for $\alpha>0$ and thus \eqref{eq:completelymonotonic} implies that
\begin{equation}
\mb{E}\|G_\theta\|^{-\alpha}_{\ell_p^n(X)\cap H_\theta} \leq \mb{E} \|G_\phi\|^{-\alpha}_{\ell_p^n(X)\cap H_\phi}
\end{equation}
for every $\alpha\in\big(0,m(n-1)\big)$ so that both expectations converge. Integration in polar coordinates now shows that
\begin{equation}
\int_{S(H_\theta)} \|u\|^{-\alpha}_{\ell_p^n(X)\cap H_\theta} \diff \sigma_{H_\theta}(u) \leq \int_{S(H_\phi)} \|u\|^{-\alpha}_{\ell_p^n(X)\cap H_\phi} \diff \sigma_{H_\phi}(u)
\end{equation}
for every $\alpha\in\big(0,m(n-1)\big)$, where $\sigma_{H_\theta}$ and $\sigma_{H_\phi}$ are the rotationally invariant probability measures on the unit spheres of $H_\theta$ and $H_\phi$ respectively. Letting $\alpha\to m(n-1)$ and using the identity
\begin{equation}
\int_{S^{\ell-1}} \|u\|_K^{-\ell} \diff\sigma(u) = \frac{|K|}{|B_2^\ell|},
\end{equation}
which holds true for every convex body $K\subseteq\R^\ell$, we finally deduce that
\begin{equation}
\big| B_{p}^n(X) \cap H_\theta\big| \leq \big|B_p^n(X) \cap H_\phi\big|.
\end{equation}
This completes the proof of Theorem \ref{thm:Lpsubspaces}. \hfill$\Box$

\medskip

\noindent {\it Theorem \ref{thm:Laplace} implies Corollary \ref{cor:meanwidth}.} Let $\theta, \phi\in S^{n-1}$ such that $(\theta_1^2,\ldots,\theta_n^2) \preceq (\phi_1^2,\ldots,\phi_n^2)$ and $X=(\R^m,\|\cdot\|_X)$ an $m$-dimensional subspace of $L_p$ whose unit ball $B_X$ is in Lewis' position. A simple inductive argument shows that for every $\beta\in(0,1]$ and $c\in(0,\infty)$, the function $h(t) = e^{-ct^\beta}$ is completely monotonic. Thus, \eqref{eq:completelymonotonic} imples that
\begin{equation} \label{eq:lasteqfornow}
\mb{E} e^{-c\|G_\theta\|^{\beta p}_{\ell_p^n(X)\cap H_\theta}} \leq \mb{E} e^{-c\|G_\phi\|^{\beta p}_{\ell_p^n(X)\cap H_\phi}}.
\end{equation}
Since both sides of \eqref{eq:lasteqfornow}, as functions of $c\in(0,\infty)$, are equal at $c=0$ we deduce that their derivatives at $c=0$ compare in the same way, that is
\begin{equation}
\mb{E} \|G_\theta\|^{\beta p}_{\ell_p^n(X)\cap H_\theta}\geq\mb{E} \|G_\phi\|^{\beta p}_{\ell_p^n(X)\cap H_\phi},
\end{equation} \label{eq:mwcompfinal}
for every $\beta\in(0,1]$. Choosing $\beta=1/p$ and integrating in polar coordinates yields
\begin{equation} \label{eq:mwcompfinal2}
\int_{S(H_\theta)} \|u\|_{\ell_p^n(X)\cap H_\theta}\diff\sigma_{H_\theta}(u) \geq \int_{S(H_\phi)} \|u\|_{\ell_p^n(X)\cap H_\phi} \diff\sigma_{H_\phi}(u).
\end{equation}
Recall that if $E=(\R^\ell,\|\cdot\|_E)$ is an $\ell$-dimensional normed space with unit ball $K$, then the mean width of the polar body $K^\circ$ is given by 
\begin{equation} \label{eq:defmw}
w(K^\circ) = \int_{S^{\ell-1}} \|u\|_{E} \diff\sigma(u).
\end{equation}
Combining \eqref{eq:defmw} with the classical polarity relation 
\begin{equation}
\big(B_p^n(X)\cap H_\theta\big)^\circ = \mathrm{Proj}_{H_\theta}\big(B_q^n(X^\ast)\big)
\end{equation}
we can rewrite \eqref{eq:mwcompfinal2} as
\begin{equation}
w\big(\mr{Proj}_{H_\theta}(B_q^n(X^\ast))\big) \geq w\big(\mr{Proj}_{H_\phi}(B_q^n(X^\ast))\big),
\end{equation}
which completes the proof of Corollary \ref{cor:meanwidth}.
\hfill$\Box$

\medskip

We now proceed with the proof of Theorem \ref{thm:Laplace}. Let $X=X_p(\mu)$ be a subspace of $L_p$, $p\in(0,2)$, whose unit ball $B_X$ is in Lewis' position. We will assume that $\mu$ is finitely supported, that is, there exist $M\in\N$, $u_1,\ldots,u_M\in S^{m-1}$ and $c_1,\ldots,c_M\in(0,\infty)$ such that $\mu = \sum_{j=1}^M c_j \delta_{u_j}$. After proving Theorem \ref{thm:Laplace} for $X=X_{p}(\mu)$ corresponding to finitely supported measures $\mu$, the general case will follow by a standard approximation argument. Notice that $\mu$ is full dimensional by the isotropicity assumption \eqref{eq:isotro}, therefore the vectors $u_1,\ldots,u_M$ span $\R^m$. 

For a block hyperplane $H_\theta$ as in \eqref{eq:Htheta} and $\e\in(0,1)$, denote
\begin{equation} \label{eq:Heps}
H_\theta(\e) = \Big\{ x=(x_1,\ldots,x_n) \in \big(\R^m\big)^n: \ \Big\|\sum_{i=1}^n \theta_i x_i\Big\|_\infty <\frac{\e}{2}\Big\}.
\end{equation}
We will make use of an identity which is a formal consequence of \cite[Lemma~14]{BGMN05}.

\begin{lemma} \label{lem:identity1}
Fix $m,n\in\N$, $p\in(0,\infty)$ and let $X=(\R^m,\|\cdot\|_X)$ be an $m$-dimensional quasi-normed space. Then for every $\lambda\in(0,\infty)$ and every unit vector $\theta$ in $\R^n$,
\begin{equation} \label{eq:identity1}
\mb{E} e^{-\lambda\|G_\theta\|^p_{\ell_{p}^{n}(X)\cap H_\theta}} =  \lim_{\e\to0^+} \frac{(2\pi)^{-\frac{m(n-1)}{2}}}{\e^m} \mu_{p,\lambda,X}^{n}\big(H_\theta(\e)\big),
\end{equation}
where
\begin{equation} \label{eq:mu}
\diff\mu_{p,\lambda,X}^n(x_1,\ldots,x_n) = e^{-\sum_{i=1}^n( \lambda \|x_i\|_X^p+\frac{1}{2}\|x_i\|_2^2)}\diff x_1 \cdots\diff x_n,
\end{equation}
and each $\diff x_i$ is Lebesgue measure on $\R^m$.
\end{lemma}

For a quasi-normed space $X=(\R^m,\|\cdot\|_X)$ which embeds into $L_p$, $p\in(0,2)$, and whose unit ball $B_X$ is in Lewis' position, the measure appearing in Lemma \ref{lem:identity1} takes the following form.

\begin{lemma} \label{lem:identity2}
Fix $m,n\in\N$, $p\in(0,2)$ and let $X=X_p(\mu)$ be an $m$-dimensional subspace of $L_p$ whose unit ball $B_X$ is in Lewis' position and $\mu=\sum_{=1}^M c_j \delta_{u_j}$. Then there exists a probability measure $\nu$ on $(0,\infty)^M$ such that
\begin{equation} \label{eq:identity2}
\frac{\diff\mu_{p,\lambda,X}^n(x_1,\ldots,x_n)}{\diff x_1 \cdots \diff x_n} = \int_{(0,\infty)^{nM}} \exp\Big(-\sum_{i=1}^n \sum_{j=1}^M s_{ij} \langle x_i, u_j\rangle^2\Big) \diff\nu_n(s),
\end{equation}
for every $(x_1,\ldots,x_n)\in \big(\R^m\big)^n$, where $\nu_n = \nu^{\otimes n}$ and $\mu_{p,\lambda,X}^n$ is defined by \eqref{eq:mu}.
\end{lemma}

\begin{proof}
By the fact that $B_X$ is in Lewis' position and the definition of $\mu$, we know that
\begin{equation}
\|x\|_X^p = \sum_{j=1}^M c_j |\langle x,u_j\rangle|^p \ \ \ \mbox{and} \ \ \ \|x\|_2^2 = \sum_{j=1}^M c_j \langle x,u_j\rangle^2,
\end{equation}
for every $x\in \R^m$. Therefore, \eqref{eq:mu} can be rewritten as
\begin{equation} \label{eq:mu'}
\frac{\diff\mu_{p,\lambda,X}^n(x_1,\ldots,x_n)}{\diff x_1 \cdots \diff x_n} = \exp\Big( - \sum_{i=1}^n \sum_{j=1}^M c_j \big( \lambda |\langle x_i, u_j\rangle|^p + \frac{1}{2}\langle x_i, u_j\rangle^2\big) \Big).
\end{equation}
A simple inductive argument shows that for every $\theta\in(0,1]$ and $c\in(0,\infty)$, the function $h(t) = e^{-ct^\theta}$ is completely monotonic. Since additionally the product of completely monotonic functions is completely monotonic and $p\in(0,2)$, we infer from Bernstein's theorem \cite{Fel71} that for every $j\in\{1,\ldots,M\}$ there exists a Borel measure $\tau_j$ on $(0,\infty)$ satisfying
\begin{equation} \label{eq:tauj}
\exp\Big(-c_j\big( \lambda t^{p/2} + \frac{1}{2} t\big)\Big) = \int_0^\infty e^{-st} \diff\tau_j(s),
\end{equation}
for every $t\in(0,\infty)$. The dominated convergence theorem easily implies that each $\tau_j$ is a probability measure. Denote $\nu \eqdef \tau_1 \otimes \cdots \otimes \tau_M$ which is also a probability measure and satisfies
\begin{equation} \label{eq:nu}
\exp\Big(-\sum_{j=1}^M c_j\big( \lambda |t_j|^{p} + \frac{1}{2} t_j^2\big)\Big) = \int_{(0,\infty)^M} \exp\Big(-\sum_{j=1}^M s_j t_j^2\Big) \diff\nu(s),
\end{equation}
for every $t_1,\ldots,t_M\in\R$. Then, combining \eqref{eq:mu'} and \eqref{eq:nu}, we deduce that
\begin{equation}
\begin{split}
\frac{\diff\mu_{p,\lambda,X}^n(x_1,\ldots,x_n)}{\diff x_1 \cdots \diff x_n} & \stackrel{\eqref{eq:mu'}}{=} \prod_{i=1}^n \exp\Big( -  \sum_{j=1}^M c_j \big( \lambda |\langle x_i, u_j\rangle|^p + \frac{1}{2}\langle x_i, u_j\rangle^2\big) \Big)
\\ & \stackrel{\eqref{eq:nu}}{=} \prod_{i=1}^n \int_{(0,\infty)^M} \exp\Big( - \sum_{j=1}^M s_{ij} \langle x_i, u_j\rangle^2\Big) \diff\nu(s_i)
\\& = \int_{(0,\infty)^{nM}} \exp\Big(-\sum_{i=1}^n \sum_{j=1}^M s_{ij} \langle x_i, u_j\rangle^2\Big) \diff\nu_n(s),
\end{split}
\end{equation}
for every $(x_1,\ldots,x_n)\in \big(\R^m\big)^n$, where $\nu_n = \nu^{\otimes n}$.
\end{proof}

Before proceeding with the proof of Theorem \ref{thm:Laplace}, we will need the following technical statement.

\begin{lemma} \label{lem:det}
Fix $m,n\in\N$ and let $M_1,\ldots,M_n$ be i.i.d. positive definite $m\times m$ random matrices. Then, for every vectors $(\alpha_1,\ldots,\alpha_n), (\beta_1,\ldots,\beta_n)\in[0,\infty)^n$ and $r\in(0,\infty)$, we have
\begin{equation} \label{eq:det}
\begin{split}
(\alpha_1,\ldots,\alpha_n)& \preceq  (\beta_1,\ldots,\beta_n) \ \ \ \Longrightarrow \ \ \  \mb{E} \Big[ \mr{det}\Big( \sum_{i=1}^n \alpha_i M_i\Big)^{-r} \Big] \leq \mb{E} \Big[\mr{det}\Big(\sum_{i=1}^n \beta_i M_i\Big)^{-r}\Big].
\end{split}
\end{equation}
\end{lemma}

\begin{proof}
Consider the function $\varphi: [0,\infty)^n\setminus\{(0,\ldots,0)\} \to \R$ given by
\begin{equation}
\varphi(\alpha_1,\ldots,\alpha_n) = \mb{E}\Big[  \mr{det}\Big( \sum_{i=1}^n \alpha_i M_i\Big)^{-r}\Big].
\end{equation}
Since $M_1,\ldots,M_n$ are i.i.d., $\varphi$ is invariant under permutations of its arguments. An elementary result of Marshall and Proschan \cite{MP65} then asserts that \eqref{eq:det} is true, provided that $\varphi$ is convex. To verify this, it suffices to check that for every positive definite matrices $A,B$ and $\lambda\in(0,1)$,
\begin{equation}
\mr{det}\big(\lambda A + (1-\lambda)B\big)^{-r} \leq \lambda \mr{det}(A)^{-r} + (1-\lambda)\mr{det}(B)^{-r}.
\end{equation}
Multiplying both sides by $\mr{det}(A)^{r}\in(0,\infty)$, this inequality can be rewritten as
\begin{equation}
\mr{det}\big(\lambda I_m + (1-\lambda) A^{-1/2} B A^{-1/2}\big)^{-r} \leq \lambda + (1-\lambda) \mr{det}\big( A^{-1/2}B A^{-1/2} \big)^{-r},
\end{equation}
which is equivalent to
\begin{equation} \label{eq:logar}
\prod_{k=1}^m \big( \lambda + (1-\lambda) \rho_k\big)^{-r} \leq \lambda + (1-\lambda) \prod_{k=1}^m \rho_k^{-r},
\end{equation}
where $\rho_1,\ldots,\rho_m\in(0,\infty)$ are the eigenvalues of the positive definite matrix $A^{-1/2}BA^{-1/2}$. To verify \eqref{eq:logar}, notice that, by the concavity of the logarithm,
\begin{equation}
\begin{split}
\log\left( \prod_{k=1}^m \big( \lambda + (1-\lambda) \rho_k\big)^{-r}\right) & = -r \sum_{k=1}^m \log\big(\lambda + (1-\lambda) \rho_k\big) \leq -r \sum_{k=1}^m (1-\lambda) \log\rho_k
\\ & = (1-\lambda) \log \prod_{k=1}^m \rho_k^{-r} \leq \log \left(  \lambda + (1-\lambda) \prod_{k=1}^m \rho_k^{-r} \right).
\end{split}
\end{equation}
This concludes the proof of \eqref{eq:det}.
\end{proof}

We are now in position to complete the proof of Theorem \ref{thm:Laplace}.

\medskip

\noindent {\it Proof of Theorem \ref{thm:Laplace}.} Suppose that $X=X_p(\mu)$, where $p\in(0,2)$ and $\mu$ is a finitely supported isotropic measure of the form $\mu=\sum_{j=1}^M c_j \delta_{u_j}$ on the unit sphere $S^{m-1}$. Then, by Lemmas \ref{lem:identity1} and \ref{lem:identity2} and Fubini's theorem,
\begin{equation} \label{eq:firstident}
\mb{E} e^{-\lambda \|G_\theta\|_{\ell_p^n(X)\cap H_\theta}^p} =  \lim_{\e\to0^+}\frac{(2\pi)^{-\frac{m(n-1)}{2}}}{\e^m} \int_{(0,\infty)^{nM}} \int_{H_\theta(\e)} \exp\Big(-\sum_{i=1}^n \sum_{j=1}^M s_{ij} \langle x_i, u_j\rangle^2\Big) \diff x \diff\nu_n(s),
\end{equation}
where $\nu_n = \nu^{\otimes n}$ for the probability measure $\nu$ on $(0,\infty)^M$ that was constructed in \eqref{eq:nu}. For fixed $s = (s_{ij})\in(0,\infty)^{nM}$, consider independent centered Gaussian random vectors $Z_1(s),\ldots,Z_n(s)$ on $\R^m$ so that $Z_i(s)$ has covariance matrix
\begin{equation} \label{eq:emai}
M_i(s) \eqdef \frac{1}{2}\Big(\sum_{j=1}^M s_{ij} u_j \otimes u_j\Big)^{-1}.
\end{equation}
Then, for every $s\in(0,\infty)^{nM}$, we have
\begin{equation}
\begin{split}
\int_{H_\theta(\e)} \exp\Big(-\sum_{i=1}^n \sum_{j=1}^M  s_{ij} \langle x_i, & u_j\rangle^2\Big) \diff x  = \int_{H_\theta(\e)} \exp\Big(- \frac{1}{2} \sum_{i=1}^n \langle M_i(s)^{-1} x_i, x_i\rangle \Big) \diff x
\\ &  = (2\pi)^{\frac{mn}{2}} \prod_{i=1}^n \sqrt{\mr{det}\big(M_i(s)\big)} \cdot \mb{P}\Big( \big(Z_1(s),\ldots,Z_n(s)\big) \in H_\theta(\e)\Big)
\end{split}
\end{equation}
and thus \eqref{eq:firstident} can be rewritten as
\begin{equation} \label{eq:secident}
\begin{split}
\mb{E} e^{-\lambda \|G_\theta\|_{\ell_p^n(X)\cap H_\theta}^p} = \lim_{\e\to0^+} \frac{(2\pi)^{m/2}}{\e^m} \int_{(0,\infty)^{nM}} \prod_{i=1}^n \sqrt{\mr{det}\big(M_i(s)\big)} \mb{P}\Big( \Big\|\sum_{i=1}^n \theta_i Z_i(s) \Big\|_\infty < \frac{\e}{2}\Big) \diff\nu_n(s).
\end{split}
\end{equation}
However, for fixed $s\in(0,\infty)^{nM}$, the Gaussian vectors $Z_1(s),\ldots,Z_n(s)$ are independent, therefore the weighted sum $\sum_{i=1}^n \theta_i Z_i(s)$ is a centered Gaussian random vector $Z(s)$ on $\R^m$ with covariance matrix
\begin{equation}
M(s) \eqdef \sum_{i=1}^n \theta_i^2 M_i(s).
\end{equation}
Therefore, continuing from \eqref{eq:secident}, if $\gamma_m$ is the standard Gaussian measure on $\R^m$, we have
\begin{equation} \label{eq:3ident}
\begin{split}
\mb{E} e^{-\lambda \|G_\theta\|_{\ell_p^n(X)\cap H_\theta}^p} & = \lim_{\e\to0^+}\frac{(2\pi)^{m/2}}{\e^m} \int_{(0,\infty)^{nM}} \prod_{i=1}^n\sqrt{\mr{det}\big(M_i(s)\big)} \cdot \mb{P}\Big( Z(s) \in \frac{\e}{2} B_\infty^m\Big) \diff\nu_n(s)
\\ & = \lim_{\e\to0^+} \frac{(2\pi)^{m/2}}{\e^m} \int_{(0,\infty)^{nM}} \prod_{i=1}^n\sqrt{\mr{det}\big(M_i(s)\big)}\cdot \gamma_m\Big(\frac{\e}{2} M(s)^{-1/2} B_\infty^m\Big) \diff\nu_n(s) 
\\ & \stackrel{(\ast)}{=}\int_{(0,\infty)^{nM}} \prod_{i=1}^n\sqrt{\mr{det}\big(M_i(s)\big)} \cdot \lim_{\e\to0^+} \frac{(2\pi)^{m/2}}{\e^m} \gamma_m \Big(\frac{\e}{2} M(s)^{-1/2}B_\infty^m\Big)\diff\nu_n(s)
\\ & \stackrel{(\dagger)}{=} \int_{(0,\infty)^{nM}} \prod_{i=1}^n\sqrt{\mr{det}\big(M_i(s)\big)} \cdot \mr{det}\Big(\sum_{i=1}^n \theta_i^2 M_i(s)\Big)^{-1/2} \diff\nu_n(s),
\end{split}
\end{equation}
where in $(\dagger)$ we used the fact that for a compact set $L\subseteq\R^m$
\begin{equation} \label{eq:42}
\lim_{\e\to0^+} \frac{(2\pi)^{m/2}}{\e^m} \gamma_m(\e L)  = \lim_{\e\to0^+} \int_L e^{-\e^2 \|w\|_2^2/2}\diff w = |L|
\end{equation}
and moreover
\begin{equation} \label{eq:43}
\Big|\frac{1}{2}M(s)^{-1/2} B_\infty^m\Big| = \mr{det}\big( M(s)^{-1/2}\big) = \mr{det}\Big(\sum_{i=1}^n \theta_i^2 M_i(s)\Big)^{-1/2}.
\end{equation}
Finally, the fact that we can swap limit and integration in $(\ast)$, follows from Lebesgue's monotone convergence theorem, since \eqref{eq:42} implies that the function $\e \mapsto \gamma_m(\e L)/\e^m$ is decreasing as $\e\to0^+$.

Consider the measure $\rho_n$ on $(0,\infty)^{nM}$ given by
\begin{equation}
\diff\rho_n(s) =\prod_{i=1}^n\sqrt{\mr{det}\big(M_i(s)\big)} \diff\nu_n(s).
\end{equation}
Since $M_i(s)$ depends only on $s_{ij}$, $j\in\{1,\ldots,M\}$ and the random matrices $M_1,\ldots,M_n$ are identically distributed with respect to the law $\nu_n = \nu^{\otimes n}$, the measure $\rho_n$ is also a product measure of the form $\rho^{\otimes n}$ for some Borel measure $\rho$ on $(0,\infty)^M$. Moreover, choosing $\theta=e_1=(1,0,\ldots,0)$ in \eqref{eq:3ident}, we deduce that
\begin{equation}
\mb{E} e^{-\lambda \|G_{e_1}\|_{\ell_p^n(X)\cap H_{e_1}}^p} = \int_{(0,\infty)^{nM}} \prod_{i=2}^n \sqrt{\mr{det}\big(M_i(s)\big)} \diff\nu_n(s) = \rho\big( (0,\infty)^M\big)^{n-1},
\end{equation}
which implies that $\rho$ is a finite measure. Therefore, \eqref{eq:3ident} can equivalently be written as
\begin{equation} \label{eq:4ident}
\mb{E}e^{-\lambda \|G_\theta\|_{\ell_p^n(X)\cap H_\theta}^p} = \int_{(0,\infty)^{nM}} \mr{det}\Big(\sum_{i=1}^n \theta_i^2 M_i(s)\Big)^{-1/2} \diff\rho_n(s).
\end{equation}
Finally, since the random matrices $M_1,\ldots,M_n$ are i.i.d. with respect to the law $\rho_n$, the conclusion of the theorem follows by combining identity \eqref{eq:4ident} with Lemma \ref{lem:det}.

To argue that \eqref{eq:Laplace} is true for a general subspace $X$ of $L_p$  (or equivalently for a general isotropic measure $\mu$ on $S^{m-1}$), we will use an approximation argument of Barthe. In \cite[pp.~55-56]{Bar04}, it was shown that for every isotropic measure $\mu$ on $S^{m-1}$ there exists a sequence $\{\mu_k\}_{k=1}^\infty$ of finitely supported isotropic measures on $S^{m-1}$ which converges weakly to $\mu$, that is, for every continuous function $f: S^{m-1}\to \R$, we have
\begin{equation}
\lim_{k\to\infty} \int_{S^{m-1}} f(\theta)\diff\mu_k(\theta) = \int_{S^{m-1}} f(\theta)\diff\mu(\theta).
\end{equation}
In particular, this means that for every $x\in\R^m$,
\begin{equation}
\lim_{k\to\infty}\|x\|_{p,\mu_k} = \lim_{k\to\infty} \Big( \int_{S^{m-1}} |\langle x,\theta\rangle |^p \diff\mu_k(\theta) \Big)^{1/p} =  \Big( \int_{S^{m-1}} |\langle x,\theta\rangle |^p \diff\mu(\theta) \Big)^{1/p} = \|x\|_{p,\mu}.
\end{equation}
If $X_k = X_p(\mu_k)$ and $X=X_p(\mu)$, then the bounded convergence theorem implies that if $\omega \in S^{m-1}$ and $G_\omega$ is a standard Gaussian random vector on the subspace $H_\omega$ of $\R^{mn}$, then for every $\lambda\in(0,\infty)$,
\begin{equation} \label{eq:barthe approximation}
\lim_{k\to\infty} \mb{E} e^{-\lambda \|G_\omega\|_{\ell_p^n(X_k)\cap H_\omega}} = \mb{E} e^{-\lambda \|G_\omega\|_{\ell_p^n(X)\cap H_\omega}}.
\end{equation}
Combining \eqref{eq:barthe approximation} with the validity of \eqref{eq:Laplace} for all spaces $X_p(\mu)$ where $\mu$ is a discrete isotropic measure, we conclude that \eqref{eq:Laplace} in fact holds for every isotropic measure $\mu$ on $S^{m-1}$.
\hfill$\Box$

\begin{remark}
In the proof of Theorem \ref{thm:Laplace}, the fact that $B_X$ is in Lewis' position was crucially used in Lemma \ref{lem:identity2} and more specifically in equation \eqref{eq:mu'}. A more general version of the result could be proven along the same lines without any assumption on the position of $B_X$, using a classical result of Levy (see, e.g., \cite[Lemma~6.4]{Kol05}) which asserts that for every $p\in(0,\infty)$, the norm of every $m$-dimensional subspace $X=(\R^m,\|\cdot\|_X)$ of $L_p$  admits a representation of the form
\begin{equation}
\|x\|_X = \Big( \int_{S^{m-1}} |\langle x,\theta\rangle|^p\diff\mu(\theta)\Big)^{1/p},
\end{equation}
for every $x\in\R^n$, where $\mu$ is a (not necessarily isotropic) finite Borel measure on $S^{m-1}$. In this case, the Gaussian random vectors $G_a$ on $H_a$ appearing in the statement of Theorem \ref{thm:Laplace} would have covariance matrices determined by the measure $\mu$.
\end{remark}

\begin{remark} \label{rem:block}
It is natural to ask whether an analogue of Theorem \ref{thm:Lpsubspaces} and in particular \eqref{eq:Lpsubspacesconsequence} can hold for sections of the form $B_p^n(X)\cap F$, where $F$ is a general subspace of $\R^{mn}$ of codimension $m$ instead of a block hyperplane of the form \eqref{eq:Htheta}. We will construct here an example of an $m$-dimensional subspace $X$ of $L_1$ whose unit ball is in Lewis' position and a coordinate subspace $F$ of $\R^{mn}$ of codimension $m$, for which
\begin{equation}
\big|B_1^n(X)\cap F\big| > \big|B_1^{n-1}(X)\big|,
\end{equation}
thus showing that the comparison \eqref{eq:Lpsubspacesconsequence} cannot hold in this generality. Let $m=4n$ and consider $X = \ell_1^{2n} \oplus_1 \ell_2^{2n}$, the space $(\R^{4n},\|\cdot\|_X)$ equipped with the norm
\begin{equation}
\|(x_1,x_2)\|_X = \|x_1\|_1 + \|x_2\|_2,
\end{equation}
where $x_1,x_2\in\R^{2n}$. Since $L_2$ admits an isometric embedding into $L_1$ (see \cite{Kad58}), $X$ is also isometric to a subspace of $L_1$ and furthermore $B_X$ is in Lewis' position. To see this, notice that
\begin{equation}
\|x\|_X = \int_{S^{2n-1}\times\{0\}}| \langle x,\theta\rangle| \diff \mu(\theta) + \int_{\{0\}\times S^{2n-1}}| \langle x,\theta\rangle| \diff\nu(\theta),
\end{equation}
where $\mu = \sum_{j=1}^{2n} \delta_{e_j}$ and $\nu$ is a multiple of the uniform measure; both these measures are isotropic. Moreover, $\ell_1^n(X)$ is, up to permutation of its coordinates, $\ell_1^{2n^2} \oplus_1 \ell_1^n(\ell_2^{2n})$ and thus there exists a coordinate subspace $F$ of codimension $m=4n$ for which $\ell_1^n(X)\cap F$ is equal to $\ell_1^{2n(n-2)} \oplus_1 \ell_1^n(\ell_2^{2n})$. We claim that
\begin{equation} \label{eq:Fbeats}
\big| B_1^n(X)\cap F\big| = \big|B_{\ell_1^{2n(n-2)}\oplus_1\ell_1^n(\ell_2^{2n})}\big| >  \big|B_{\ell_1^{2n(n-1)}\oplus_1\ell_1^{n-1}(\ell_2^{2n})}\big| = \big|B_1^{n-1}(X)\big|.
\end{equation}
To show \eqref{eq:Fbeats}, we will use the fact that for any $k$-dimensional normed space $Z=(\R^k,\|\cdot\|_Z)$,
\begin{equation}
\begin{split}
\frac{1}{k!} \int_{\R^k} e^{-\|z\|_Z}\diff z = \frac{1}{k!} \int_{\R^k} \int_{\|z\|_Z}^\infty e^{-s}\diff s \diff z = \frac{1}{k!} \int_0^\infty  \int_{sB_Z} & e^{-s}\diff z\diff s \\ & = \frac{|B_Z|}{k!} \int_0^\infty s^k e^{-s} \diff s = |B_Z|,
\end{split}
\end{equation}
which immediately implies that if $Z_1, Z_2$ are of dimensions $k_1$ and $k_2$ respectively, then
\begin{equation} \label{eq:voldirsum}
\begin{split}
\big|B_{Z_1\oplus_1 Z_2}\big| & = \frac{1}{(k_1+k_2)!} \int_{\R^{k_1}} \int_{\R^{k_2}} e^{-\|z_1\|_{Z_1} - \|z_2\|_{Z_2}}\diff z_2 \diff z_1 
\\ & = \frac{1}{(k_1+k_2)!} \int_{\R^{k_1}} e^{-\|z_1\|_{Z_1}}\diff z_1 \int_{\R^{k_2}} e^{-\|z_2\|_{Z_2}}\diff z_2 = \frac{k_1!k_2!}{(k_1+k_2)!} \big|B_{Z_1}\big| \big|B_{Z_2}\big|.
\end{split}
\end{equation}
Moreover, an iteration of \eqref{eq:voldirsum} shows that if $Z=(\R^k,\|\cdot\|_Z)$ and $s\in\N$, then
\begin{equation} \label{eq:volpower}
\big| B_1^s(Z)\big| = \frac{(k!)^s}{(sk)!} \big| B_Z\big|^s.
\end{equation}
Hence, for $a,b,c\in\N$, we have
\begin{equation}
\big|B_{\ell_1^a \oplus_1 \ell_1^b(\ell_2^c)}\big| \stackrel{\eqref{eq:voldirsum}}{=} \frac{a!(bc)!}{(a+bc)!} |B_1^a| |B_1^b(\ell_2^c)| \stackrel{\eqref{eq:volpower}}{=} \frac{2^a (c!)^b}{(a+bc)!} \cdot \frac{\pi^{bc/2}}{\Gamma\big(\frac{c}{2}+1\big)^b},
\end{equation}
using the well-known formulas for the volumes of $B_1^a$ and $B_2^c$. Therefore, \eqref{eq:Fbeats} is equivalent to
\begin{equation}
\frac{2^{2n(n-2)} \big((2n)!\big)^n \pi^{n^2}}{\big(4n(n-1)\big)! (n!)^n} > \frac{2^{2n(n-1)} \big((2n)!\big)^{n-1} \pi^{n(n-1)}}{\big(4n(n-1)\big)! (n!)^{n-1}}
\end{equation}
which can be rewritten as
\begin{equation} \label{eq:lastremark}
\frac{(2n)!}{n!} > \Big(\frac{4}{\pi}\Big)^n.
\end{equation}
Finally, to verify \eqref{eq:lastremark} notice that
\begin{equation}
\frac{(2n)!}{n!} = \prod_{i=1}^n (n+i) \geq 2^n > \Big(\frac{4}{\pi}\Big)^n
\end{equation}
and \eqref{eq:Fbeats} follows.
\end{remark}

\section{Proof of Proposition \ref{prop:liako}} \label{sec:3}

We now proceed with the proof of Proposition \ref{prop:liako}. The argument relies on a recent observation of Liakopoulos \cite{Lia18}, who extended  Ball's version \cite{Bal91a} of the classical Loomis--Whitney inequality \cite{LW49}. In \cite{Lia18}, the author showed that if the subspaces $F_1,\ldots,F_r$ of $\R^k$ induce a decomposition of the identity of the form
\begin{equation}
s I_k = \sum_{i=1}^r c_i \mathrm{Proj}_{F_i},
\end{equation}
where $\mathrm{Proj}_{F_i}$ is the orthogonal projection on $F_i$ and $ c_i\in(0,\infty)$, then for every compact set $L\subseteq \R^k$,
\begin{equation} \label{eq:liako}
|L|^s \leq \prod_{i=1}^r \big| \mathrm{Proj}_{F_i}(L)\big|^{c_i}.
\end{equation}
In \cite{Bal91a}, Ball showed the above implication when all the $F_i$ are hyperplanes using his geometric Brascamp--Lieb inequality from \cite{Bal89} and Liakopoulos' proof of the general case proceeds along the same lines using Barthe's \cite{Bar98} multidimensional geometric Brascamp--Lieb inequality.

\medskip

\noindent {\it Proof of Proposition \ref{prop:liako}.}
For $(\e_1,\ldots,\e_n)\in\{-1,1\}^n$ and a permutation $\sigma \in S_n$, denote
\begin{equation}
H_\theta^{\e,\sigma} \eqdef \Big\{ (x_1,\ldots,x_n)\in\big(\R^m\big)^n: \ \sum_{i=1}^n \e_i \theta_{\sigma(i)} x_i=0\Big\}.
\end{equation}
Moreover, let $P_{\e,\sigma} \eqdef \mathrm{Proj}_{H_\theta^{\e,\sigma}}$ be the orthogonal projection on $H_\theta^{\e,\sigma}$. An elementary computation shows that
\begin{equation}
P_{\e,\sigma}(x_1,\ldots,x_n) = \Big( x_1 - \e_1\theta_{\sigma(1)} \sum_{k=1}^n \e_k\theta_{\sigma(k)} x_k, \ldots, x_n - \e_n \theta_{\sigma(n)} \sum_{k=1}^n \e_k\theta_{\sigma(k)} x_k\Big),
\end{equation}
for every $(x_1,\ldots,x_n)\in \big(\R^m\big)^n$. Averaging over $\e\in\{-1,1\}^n$, we get
\begin{equation}
\begin{split}
\frac{1}{2^n}\sum_{\e\in\{-1,1\}^n} P_{\e,\sigma}(x_1,\ldots, x_n)  = \Big( x_i - \frac{1}{2^n} \sum_{\e\in\{-1,1\}^n} & \e_i\theta_{\sigma(i)}\sum_{k=1}^n \e_k\theta_{\sigma(k)} x_k\Big)_{i=1}^n
 \\ & = \big( x_1 - \theta_{\sigma(1)}^2x_1,\ldots, x_n - \theta_{\sigma(n)}^2 x_n\big),
\end{split}
\end{equation}
which after further averaging over $\sigma \in S_n$ becomes
\begin{equation} \label{eq:avgsigma}
\begin{split}
\frac{1}{2^n n!} \sum_{\e\in\{-1,1\}^n} \sum_{\sigma\in S_n} P_{\e,\sigma}(x_1,\ldots,x_n)  = \Big( & x_i -  \frac{1}{n!} \sum_{\sigma\in S_n}  \theta_{\sigma(i)}^2 x_i \Big)_{i=1}^n \\ & = \Big( x_i - \frac{\theta_1^2+\cdots+\theta_n^2}{n} x_i \Big)_{i=1}^n 
 = \frac{n-1}{n} (x_1,\ldots,x_n).
\end{split}
\end{equation}
Equation \eqref{eq:avgsigma} can be rewritten as
\begin{equation}
\frac{1}{2^n n!} \sum_{\e\in\{-1,1\}^n} \sum_{\sigma\in S_n} P_{\e,\sigma} = \frac{n-1}{n} I_{mn},
\end{equation}
hence, by \eqref{eq:liako} applied to the compact set $L=K^n$, we have that
\begin{equation} \label{eq:useliako}
|K|^{n-1} = \big|K^n\big|^{\frac{n-1}{n}} \leq \prod_{\e\in\{-1,1\}^n} \prod_{\sigma\in S_n} \big|P_{\e,\sigma}(K^n)\big|^{\frac{1}{2^nn!}}.
\end{equation}
However, notice that $(x_1,\ldots,x_n)\in K^n$ if and only if $(\e_1 x_{\sigma(1)}, \ldots, \e_n x_{\sigma(n)}) \in K^n$ for every $\e\in\{-1,1\}^n$ and $\sigma \in S^n$ since $K$ is centrally symmetric. Therefore $|P_{\e,\sigma}(K^n)| = |\mr{Proj}_{H_\theta}(K^n)|$ for every $\e\in\{-1,1\}^n$ and $\sigma\in S_n$ and thus \eqref{eq:useliako} can equivalently be written as
\begin{equation}
|K|^{n-1} \leq \prod_{\e\in\{-1,1\}^n} \prod_{\sigma\in S_n} \big|\mathrm{Proj}_{H_\theta}(K^n)\big|^{\frac{1}{2^nn!}} = \big|\mathrm{Proj}_{H_\theta}(K^n)\big|,
\end{equation}
which completes the proof of Proposition \ref{prop:liako}.
\hfill$\Box$

\section{Concluding remarks} \label{sec:4}

\subsection*{1} It follows from Theorem \ref{thm:Lpsubspaces} that for every $p\in(0,2]$, $X=(\R^m,\|\cdot\|_X)$ which admits an isometric embedding into $L_p$ and unit vector $\theta$ in $\R^n$,
\begin{equation} \label{eq:quest2}
\big|B_p^n(X) \cap H_\theta\big| \leq \big| B_p^{n-1}(X)\big|.
\end{equation}
It is conceivable that the reverse inequality of \eqref{eq:quest2} holds true when $p>2$.

\begin{question} \label{q:p>2}
Let $p\in(2,\infty]$ and $X=(\R^m,\|\cdot\|_X)$ which is isometric to a subspace of $L_p$. Is it true that
\begin{equation} \label{eq:q69}
\big| B_p^{n-1}(X)\big| \leq \big| B_p^n(X)\cap H_\theta\big|
\end{equation}
for every unit vector $\theta$ in $\R^n$? In particular, does every symmetric convex (or even star) body $K\subseteq\R^m$ satisfy
\begin{equation} \label{eq:q70}
|K|^{n-1} \leq \big|K^n \cap H_\theta\big|
\end{equation}
for every unit vector $\theta$ in $\R^n$?
\end{question}
\noindent Inequality \eqref{eq:q69} for $X=\ell_q^m$, where $q\in[2,p]$, follows from \cite[Theorem~19]{Bar01}. It was shown there that for every normed space $X=(\R^m,\|\cdot\|_X)$ and every subspace $F$ of $\R^{mn}$ of dimension $mk$, the function
$$(0,\infty]\ni r \longmapsto \frac{|B_r^n(X)\cap F|}{|B_r^{k}(X)|}$$
is nondecreasing. Therefore, if $p\geq q\geq2$,
$$ \frac{|B_p^n(\ell_q^m)\cap H_\theta|}{|B_p^{n-1}(\ell_q^m)|} \geq \frac{|B_q^n(\ell_q^m)\cap H_\theta|}{|B_q^{n-1}(\ell_q^m)|} = \frac{|B_q^{mn}(\R)\cap H_\theta|}{|B_q^{m(n-1)}(\R)|} \geq 1,$$
where the last inequality follows from \cite{MP88}. In particular, \eqref{eq:q70} holds true for $K=B_q^m(\R)$ for every $q\in[2,\infty]$.

\subsection*{2} Even though every separable Banach space embeds isometrically in $L_\infty$, an analogue of the results of \cite{Bal86} and \cite{OP00} cannot hold for general $B_\infty^n(X)$ spaces. It has been shown by Brzezinski \cite{Brz13} that for every $n\in \N$, $m\geq2$ and unit vector $\theta$ in $\R^n$,
\begin{equation} \label{eq:brz}
\frac{\big|(B_2^m)^n \cap H_\theta\big|}{|B_2^m|^{n-1}} \leq \lim_{k\to\infty} \frac{\big| (B_2^m)^k \cap H_{\mr{diag}}\big|}{|B_2^m|^{k-1}} = \frac{(2m+4)^{\frac{m}{2}}}{m2^{m-1}\Gamma\big(\frac{m}{2}\big)},
\end{equation}
which shows that sections of $B_\infty^n(\ell_2^m)$ with the diagonal block hyperplane $H_\mr{diag}$ maximize the volume {\it asymptotically} in $n$ when $X=\ell_2^m$ and $m\geq2$ (after proper normalization). We note that the upper bound of \eqref{eq:brz} coincides with the bound of Oleszkiewicz and Pe\l czy\'nski  \cite{OP00} when $m=2$. This should be viewed in contrast with Ball's theorem \cite{Bal86}, according to which $B_\infty^n(\R)\cap\big(\frac{1}{\sqrt{2}},\frac{1}{\sqrt{2}},\ldots,0\big)$ is the maximal section of the unit cube $B_\infty^n(\R)$.

We note in passing that a formal strengthening of Brzezinski's result can be obtained using an estimate of Gluskin and Milman. It follows from \cite[Proposition~2]{GM04} that for every compact set $A \subseteq\R^m$ such that $|A| = |B_2^m|$ and unit vector $\theta$ in $\R^n$, we have
\begin{equation}
\big| A^n \cap H_\theta\big| \leq \big| (B_2^m)^n \cap H_\theta\big|,
\end{equation}
which combined with \eqref{eq:brz} shows that for every compact set $A\subseteq\R^m$, $m\geq2$, 
\begin{equation}
\big|A^n\cap H_\theta\big| \leq \frac{(2m+4)^{\frac{m}{2}}|A|^{n-1}}{m2^{m-1}\Gamma\big(\frac{m}{2}\big)}
\end{equation}
and, by \cite{Bal86}, $|A^n \cap \theta^\perp| \leq \sqrt{2}|A|^{n-1}$ for compact sets $A\subseteq\R$.

\subsection*{3} Choosing $X=\ell_p^m$, $p\in(0,2)$ in the statement of Theorem \ref{thm:Lpsubspaces} we deduce that
\begin{equation} \label{eq:lowdim}
\big| B_p^{mn}(\R)\cap H_\theta\big| = \big| B_p^n(\ell_p^m) \cap H_\theta\big| \geq \big| B_p^n(\ell_p^n) \cap H_\mr{diag}\big| = \big| B_p^{mn}(\R)\cap H_\mr{diag}\big|,
\end{equation}
for every unit vector $\theta$ in $\R^n$. This observation is relevant to the following well-known open question.

\begin{question} \label{q:lowdim}
Fix $p\in(0,2)$ and $k,d \in\N$ satisfying $d\leq k-2$. What are the minimal $d$-dimensional sections of $B_p^k(\R)$?
\end{question}

\noindent Equation \eqref{eq:lowdim} asserts that when $k$ is a multiple of $d$, the minimal $d$-dimensional section of $B_p^k(\R)$ with a {\it block} subspace is $B_p^k(\R)\cap H_\mr{diag}$. In contrast to the case of the unit cube (see \cite{Bal89}), these sections do not extremize the volume of $B_p^k(\R)\cap E$ over a general $d$-dimensional subspace $E$ when $p\in(0,2)$. As was pointed out to me by an anonymous referee, if $k=4$, $d=2$ and
\begin{equation}
E = \mathrm{span}\Big\{ \Big(1,\frac{1}{\sqrt{2}}, 0, - \frac{1}{\sqrt{2}}\Big), \Big(0, \frac{1}{\sqrt{2}}, 1, \frac{1}{\sqrt{2}}\Big)\Big\},
\end{equation}
then
\begin{equation}
\big| B_1^4 \cap H_{\mr{diag}}\big| = 1 > 4(3\sqrt{2}-4) =  \big|B_1^4\cap E\big|.
\end{equation}

\subsection*{4} A symmetric convex body $K\subseteq \R^n$ is called 1-symmetric if $(x_1,\ldots,x_n)\in K$ if and only if for every $\e\in\{-1,1\}^n$ and $\sigma\in S_n$ also $(\e_1x_{\sigma(1)},\ldots,\e_nx_{\sigma(n)}) \in K$. A slight variant of the proof of Proposition \ref{prop:liako}, yields the following estimate on projections of 1-symmetric bodies.

\begin{proposition} \label{prop:1sym}
Fix $d, n\in\N$ with $d\in\{1,\ldots,n\}$. For every 1-symmetric convex body $K\subseteq\R^n$ and $d$-dimensional subspace $F$ of $\R^n$, we have
\begin{equation}
\big| \mr{Proj}_F K \big| \geq |K|^\frac{d}{n}.
\end{equation}
\end{proposition}

\noindent We shall sketch the proof of Proposition \ref{prop:1sym}. Let $F$ be a $d$-dimensional subspace of $\R^n$. For $\e\in\{-1,1\}^n$ and a permutation $\sigma\in S_n$ consider the subspace
\begin{equation}
F^{\e,\sigma} = \big\{ (\e_1 x_{\sigma(1)},\ldots, \e_n x_{\sigma(n)})\in\R^n: \ (x_1,\ldots,x_n)\in F\big\}
\end{equation}
and denote by $P_{\e,\sigma}$ the orthogonal projection on $F^{\e,\sigma}$. As in the proof of Proposition \ref{prop:liako}, one can easily observe that every 1-symmetric convex body $K\subseteq\R^n$ satisfies
\begin{equation} \label{eq:symmetrydone}
| P_{\e,\sigma} K | = \big| \mr{Proj}_F K\big|,
\end{equation}
for every $\e\in\{-1,1\}^n$ and $\sigma \in S_n$. Moreover, a simple algebraic computation using bases shows that
\begin{equation} \label{eq:algident}
\frac{1}{2^n n!} \sum_{\e\in\{-1,1\}^n} \sum_{\sigma\in S_n} P_{\e,\sigma} = \frac{d}{n} I_n.
\end{equation} 
Therefore, invoking inequality \eqref{eq:liako}, we deduce that
\begin{equation}
|K|^{\frac{d}{n}} \stackrel{\eqref{eq:liako}\wedge\eqref{eq:algident}}{\leq} \prod_{\e\in\{-1,1\}^n} \prod_{\sigma\in S_n} |P_{\e,\sigma} K|^{\frac{1}{2^nn!}} \stackrel{\eqref{eq:symmetrydone}}{=} \prod_{\e\in\{-1,1\}^n} \prod_{\sigma\in S_n} \big|\mr{Proj}_F  K\big|^{\frac{1}{2^nn!}} = \big|\mr{Proj}_F  K\big|,
\end{equation}
for every 1-symetric convex body $K\subseteq\R^n$, thus concluding the proof of Proposition \ref{prop:1sym}.

We note that for $d=n-1$, Proposition \ref{prop:1sym} follows from the main result of \cite{Bal91a}, where Ball proved that every convex body $K\subseteq\R^n$ whose projection body $\Pi K$ is in John's position (see, e.g., \cite{AAGM15}) satisfies
\begin{equation} \label{eq:ballsineq}
\big|\mr{Proj}_{\theta^\perp} K\big| \geq|K|^{\frac{n-1}{n}},
\end{equation}
for every unit vector $\theta$ in $\R^n$. Notice that if $K$ is 1-symmetric, then the same holds true for its projection body $\Pi K$. Therefore, combining Ball's theorem with the well-known fact that every 1-symmetric convex body can be rescaled to be in John's position, we conclude that \eqref{eq:ballsineq} holds for every 1-symmetric convex body $K\subseteq\R^n$ and every unit vector $\theta$ in $\R^n$.

\bibliography{mixedlp}
\bibliographystyle{alpha}
\nocite{*}

\end{document}